\documentclass[12pt]{article}
\usepackage{amsmath,amssymb,amsthm}
\newtheorem{thm}{Theorem}[section]

 \newtheorem{cor}{Corollary}[section]
 \newtheorem{lem}{Lemma}[section]
 \newtheorem{prop}{Proposition}[section]
 \newtheorem{defn}{Definition}[section]%是否是全文计数？
\newtheorem{rem}{Remark}[section]

\begin{document}
\begin{center}
{\large{\bf The frequency-localization technique and
minimal decay-regularity for Euler-Maxwell equations}}
\end{center}
\par
\begin{center}
\footnotesize{Jiang Xu}\\[2ex]
\footnotesize{Department of Mathematics, \\ Nanjing
University of Aeronautics and Astronautics, \\
Nanjing 211106, P.R.China,}\\
\footnotesize{jiangxu\underline{ }79@nuaa.edu.cn}\\
\vspace{10mm}

\footnotesize{Shuichi Kawashima}\\[2ex]
\footnotesize{Faculty of Mathematics, \\ Kyushu University, Fukuoka 819-0395, Japan,}\\
\footnotesize{kawashim@math.kyushu-u.ac.jp}\\
\end{center}
\vspace{6mm}

\begin{abstract}
Dissipative hyperbolic systems of \textit{regularity-loss} have been recently received increasing attention.
Usually, extra higher regularity is assumed to obtain the optimal decay estimates, in comparison with that
for the global-in-time existence of solutions.
In this paper, we develop a new frequency-localization time-decay property, which enables us to
overcome the technical difficulty and improve the minimal decay-regularity for dissipative systems. As an application, it is shown that the optimal decay rate of $L^1(\mathbb{R}^3)$-$L^2(\mathbb{R}^3)$ is available for Euler-Maxwell equations with the critical regularity $s_{c}=5/2$, that is, the extra higher regularity is not needed.
\end{abstract}


\section{Introduction}
In this paper, we are interested in compressible isentropic Euler-Maxwell equations in plasmas physics (see, for example, \cite{CJW,J}), which are given in the form
\begin{equation}
\left\{
\begin{array}{l}\partial_{t}n+\nabla\cdot(nu)=0,\\
 \partial_{t}(nu)+\nabla\cdot(nu\otimes u)+\nabla p(n)=-n(E+u\times
B)-nu,\\
\partial_{t}E-\nabla\times B=
nu,\\
\partial_{t}B+\nabla\times E=0,\\
 \end{array} \right.\label{R-E1}
\end{equation}
with constraints
\begin{equation}
\nabla\cdot E=n_{\infty}-n,\ \ \ \nabla\cdot B=0 \label{R-E2}
\end{equation}
for $(t,x)\in[0,+\infty)\times\mathbb{R}^{3}$.
Here the unknowns
$n>0, u\in\mathbb{R}^{3}$ are the density and the velocity of electrons,
and $E\in \mathbb{R}^{3}, B\in \mathbb{R}^{3}$
denote the electric field and magnetic field, respectively. The pressure $p(n)$ is a
given smooth function of $n$ satisfying $p'(n)>0$ for $n>0$. For the sake of simplicity, $n_{\infty}$ is assumed to be
a positive constant, which
stands for the density of positively charged background ions. Observe that system (\ref{R-E1}) admits a constant equilibrium state
$(n_{\infty},0,0,B_{\infty})$, which is regarded as vector in $\mathbb{R}^{10}$, where $B_{\infty}\in \mathbb{R}^{3}$ is an arbitrary fixed constant vector.
The main objective of this paper is to establish the large-time behavior for the corresponding Cauchy problem. For this purpose,  system (\ref{R-E1}) is supplemented with the initial data
\begin{equation}
(n,u,E,B)|_{t=0}=(n_{0},u_{0},E_{0},B_{0})(x),\ \ x\in \mathbb{R}^{3}. \label{R-E3}
\end{equation}
Clearly, it is not difficult to see that (\ref{R-E2}) can hold for any $t>0$ if the initial data satisfy the following compatible conditions
\begin{equation}
\nabla\cdot E_{0}=n_{\infty}-n_{0},\ \ \
\nabla\cdot B_{0}=0, \ \ x\in\mathbb{R}^{3}. \label{R-E4}
\end{equation}
Set $w=(n,u,E,B)^{\top}$ ($\top$ transpose) and $w_{0}=(n_{0},u_{0},E_{0},B_{0})^{\top}$.
Then (\ref{R-E1}) can be be written in the vector form
\begin{equation}
A^{0}(w)w_{t}+\sum_{j=1}^{3}A^{j}(w)w_{x_{j}}+L(w)w=0,\label{R-E5}
\end{equation}
where the coefficient matrices are given explicitly as
\begin{eqnarray*}
A^{0}(w)=\left(
           \begin{array}{cccc}
             a(n) & 0 & 0 & 0 \\
             0 & nI & 0 & 0 \\
             0 & 0 & I & 0 \\
             0 & 0 & 0 & I \\
           \end{array}
         \right),\ L(w)=\left(
       \begin{array}{cccc}
         0 & 0 & 0 & 0 \\
         0 & n(I-\Omega_{B}) & nI & 0 \\
         0 & -nI & 0 & 0 \\
         0 & 0 & 0 & 0 \\
       \end{array}
     \right),
     \end{eqnarray*}
\begin{eqnarray*}
\sum_{j=1}^{3}A^{j}(w)\xi_{j}=\left(
                                \begin{array}{cccc}
                                  a(n)(u\cdot\xi) & p'(n)\xi & 0 & 0 \\
                                  p'(n)\xi^{\top} & n(u\cdot\xi)I & 0 & 0 \\
                                  0 & 0 & 0 & -\Omega_{\xi} \\
                                  0 & 0 & \Omega_{\xi} & 0 \\
                                \end{array}
                              \right).
\end{eqnarray*}
Here, $a(n):=p'(n)/n$ is the enthalpy function, $I$ is the identity matrix of third order. For any $\xi=(\xi_{1},\xi_{2},\xi_{3})\in \mathbb{R}^3$, $\Omega_{\xi}$ is the skew-symmetric matrix defined by
\begin{eqnarray*}
\Omega_{\xi}=\left(
               \begin{array}{ccc}
                 0 & -\xi_{3} & \xi_{2} \\
                 \xi_{3} & 0 & -\xi_{1} \\
                 -\xi_{2} & \xi_{1} & 0 \\
               \end{array}
             \right)
\end{eqnarray*}
such that $\Omega_{\xi} E^{\top}=(\xi\times E)^{\top}$ (as a column vector in $\mathbb{R}^{3}$).

Clearly, (\ref{R-E5}) is a symmetric hyperbolic system, since
$A^{0}(w)$ is real symmetric and positive definite and $A^{j}(w)(j=1,2,3)$ are real symmetric. Generally, the main feature of (\ref{R-E5}) is the finite time blowup
of classical solutions even when the initial data are smooth and small. In one dimensional space, Chen, Jerome and Wang \cite{CJW} first constructed global weak solutions by using the Godunov scheme of the fractional step. By using the dissipative effect of damping terms, Peng Wang and Gu \cite{PWG} established the global existence of smooth solutions in the periodic domain. Duan \cite{D1} analyzed the regularity-loss mechanism in the dissipation part and constructed the global existence and time-decay estimates of smooth solutions. The first author \cite{X} made the best use of the coupling structure of each equation in (\ref{R-E5}) and constructed the global classical solutions in spatially critical Besov spaces. So far there are a number of efforts on the Euler-Maxwell system (\ref{R-E1}) with or without dissipation, see \cite{DLZ,GIP,GM,P1,P2,TWW,UK,UWK,XXK} and therein references.

For the Cauchy problem (\ref{R-E1})-(\ref{R-E3}), in this paper, we focus on the quantitative decay estimates of solutions toward the equilibrium state $w_{\infty}=(n_{\infty},0,0,B_{\infty})$. Ueda, Wang and the second author \cite{UWK} studied the dissipative structure of (\ref{R-E5}), which is weaker than the standard one characterized in \cite{BHN,HN,Ka1,KY,KY2,SK,UKS,XK1,XK2,Y}. More precisely, the dissipative matrix $L(w)$ is nonnegative definite, however, $L(w)$ is not real symmetric, which leads to the regularity-loss not only in the dissipation part of the energy estimate but also in the decay estimate for the linearized system. To clarify it, let us reformulate (\ref{R-E1}) as
the linearized perturbation form around the equilibrium state $w_{\infty}$:
\begin{equation}
\left\{
\begin{array}{l}\partial_{t}n+n_{\infty}\mathrm{div}\upsilon=0,\\
\partial_{t}\upsilon+a_{\infty}\nabla n+E+\upsilon\times B+\upsilon=(\mathrm{div}q_{2}+r_{2})/n_{\infty},
\\
\partial_{t}E-\nabla\times B-n_{\infty}\upsilon=0,\\
\partial_{t}B+\nabla\times E=0,\\
 \end{array} \right.\label{R-E6}
\end{equation}
where $\upsilon=nu/n_{\infty},\  a_{\infty}=p'(n_{\infty})/n_{\infty}$,
$$q_{2}=-n_{\infty}^2\upsilon\otimes \upsilon/n-[p(n)-p(n_{\infty})-p'(n_{\infty})(n-n_{\infty})]I $$
and
$$r_{2}=-(n-n_{\infty})E-n_{\infty}\upsilon\times(B-B_{\infty}).$$
We put $z:=(\rho, \upsilon, E, h)^{\top}$, where $\rho=n-n_{\infty}$ and $h=B-B_{\infty}$. The corresponding initial data are given by
\begin{equation}
z|_{t=0}=(\rho_{0}, \upsilon_{0}, E_{0}, h_{0})^{\top}(x) \label{R-E7}
\end{equation}
with $\rho_{0}=n_{0}-n_{\infty},\ \upsilon_{0}=n_{0}u_{0}/n_{\infty}$ and $h_{0}=B_{0}-B_{\infty}$.
System (\ref{R-E6}) can be also rewrite in the vector form as
\begin{equation}
A^{0}z_{t}+\sum_{j=1}^{3}A^{j}z_{x_{j}}+Lz=\sum_{j=1}^{3}Q_{x_{j}}+R, \label{R-E8}
\end{equation}
where $A^{0}, A^{j} $ and $L$ are the constant matrices in (\ref{R-E5}) with $w=w_{\infty}$, $Q(z)=(0, q^{j}_{2}/n_{\infty}, 0, 0)^{\top}$ and $R(z)=(0,r_{2}/n_{\infty}, 0,0)^{\top}$. Observe that $Q(z)=O(|(\rho,\upsilon)|^2)$ and $R(z)=O(\rho |E|+|\upsilon||h|)$.
The linearized form of (\ref{R-E8}) reads as
\begin{equation}
A^{0}\partial_{t}z_{\mathcal{L}}+\sum_{j=1}^{3}A^{j}\partial_{x_{j}}z_{\mathcal{L}}+Lz_{\mathcal{L}}=0, \label{R-E9}
\end{equation}
and the corresponding initial data $z_{0}:=(\rho_{0}, \upsilon_{0}, E_{0}, h_{0})^{\top}$ satisfy
\begin{equation}
\mathrm{div}E_{0}=-\rho_{0},\ \ \ \mathrm{div}h_{0}=0. \label{R-E10}
\end{equation}
In \cite{UK}, Ueda and the second author showed that the Fourier image of $z_{\mathcal{L}}$ satisfies the following pointwise estimate
\begin{eqnarray}
|\widehat{z_{\mathcal{L}}}(t,\xi)|\lesssim e^{-c_{0}\eta_{0}(\xi)t} |\hat{z}_{0}|\label{R-E11}
\end{eqnarray}
for any $t\geq0$ and $\xi\in \mathbb{R}^{3}$, where the dissipative rate $\eta_{0}(\xi)=|\xi|^2/(1+|\xi|^2)^2$ and $c_{0}>0$ is a constant.
Furthermore, the decay estimate of $z_{\mathcal{L}}$ was followed:
\begin{eqnarray}
\|\partial^{k}_{x}z_{\mathcal{L}}\|_{L^2}\lesssim (1+t)^{-3/4-k/2}\|z_{0}\|_{L^1}+(1+t)^{-\ell/2}\|\partial^{k+\ell}_{x}z_{0}\|_{L^2}, \label{R-E12}
\end{eqnarray}
where $k$ and $\ell$ are non-negative integers.

\begin{rem}\label{rem1.2}
The decay (\ref{R-E12}) is of the regularity-loss type, since $(1+t)^{-\ell/2}$ is created by assuming the additional $\ell$-th order regularity on the initial data.
As a matter of fact, similar dissipative mechanisms also appear in the study of other dissipative systems which were investigated by the second author and his collaborators in recent several years, for instance, Timoshenko systems in \cite{IHK,IK,LK}, hyperbolic-elliptic systems of radiating gas in \cite{HK}, a plate equation with rotational inertia effect in \cite{SK2}, hyperbolic systems of viscoelasticity in \cite{Dh,DNK}, as well as the Vlasov-Maxwell-Boltzmann system studied by Duan and Strain et al. (see for example, \cite{D2,DS}).
\end{rem}

Furthermore, based on the decay estimate (\ref{R-E12}) of linearized solutions, the decay estimates of Euler-Maxwell equations (\ref{R-E6})-(\ref{R-E7}) can be obtained by a combination of the time weighted energy method and the semigroup approach. To overcome the major difficulty arising from the weaker mechanism of regularity-loss, the decay regularity index is usually needed to be sufficiently large, for instance, $s\geq 13$ in \cite{D1}, $s\geq 6$ in \cite{UK} and therein references. Very recently, Tan, Wang and Wang \cite{TWW} obtained various decay rates of the solution and its derivatives for (\ref{R-E1})-(\ref{R-E3})
by a regularity interpolation trick. Actually, their decay results may be available in the price of higher regularity of initial data. In the present paper, we investigate a different but interesting problem in comparison with previous efforts. That is, which regularity index does characterize
the minimal decay regularity for (\ref{R-E1})-(\ref{R-E3})? For this motivation, we formulate a definition on the ``minimal decay regularity".
\begin{defn}\label{defn1.1}
If the optimal decay rate of $L^{1}(\mathbb{R}^n)$-$L^2(\mathbb{R}^n)$ type is achieved under the lowest regularity assumption, then the lowest index is called the minimal decay regularity for dissipative systems of regularity-loss, which is labelled as $s_{D}$.
\end{defn}
Obviously, following from Definition \ref{defn1.1}, we found $s_{D}\leq13$ in \cite{D1} and $s_{D}\leq6$ in \cite{UK}. Based on the recent works \cite{X,XXK},
in Besov space with the regularity $5/2$, is it possible to get optimal decay rates in functional spaces with relatively lower regularity?
The current paper attempts to seek the minimal decay regularity for (\ref{R-E1})-(\ref{R-E3}) such that $s_{D}\leq5/2$. To achieve it, actually, we need a new frequency-localization time-decay inequality for the dissipative rate $\eta_{0}(\xi)=|\xi|^2/(1+|\xi|^2)^2$. On the other hand,
other dissipative systems of regularity-loss as in \cite{Dh,DNK,D2,DS,HK,IHK,IK,LK,SK2} have been recently investigated, where the dissipative rate is subjected to the $(a,b)$-type: $\eta(\xi)=|\xi|^{2a}/(1+|\xi|^2)^{b}$. Therefore, we are going to develop the following general frequency-localization time-decay inequality for forthcoming possible works.

\begin{thm}\label{thm1.1}
Let $\eta(\xi)$ be a positive, continuous and real-valued function in $\mathbb{R}^{n}$ satisfying
\begin{eqnarray}\eta(\xi)\sim\left\{
                 \begin{array}{ll}
                   |\xi|^{\sigma_{1}}, &  |\xi|\rightarrow 0; \\
                   |\xi|^{-\sigma_{2}}, & |\xi|\rightarrow\infty;
                 \end{array}
               \right. \label{R-E13}
\end{eqnarray}
for $\sigma_{1}, \sigma_{2}>0$.

If $f\in \dot{B}^{s+\ell}_{r,\alpha}(\mathbb{R}^{n})\cap \dot{B}^{-\varrho}_{2,\infty}(\mathbb{R}^{n})$ for $s\in \mathbb{R}, \varrho\in \mathbb{R}$ and $1\leq \alpha\leq\infty$ such that $s+\varrho>0$, then it holds that
\begin{eqnarray}
&&\Big\|2^{qs}\|\widehat{\dot{\Delta}_{q}f}e^{-\eta(\xi)t}\|_{L^{2}}\Big\|_{l^{\alpha}_{q}}\nonumber\\ &\lesssim & \underbrace{(1+t)^{-\frac{s+\varrho}{\sigma_{1}}}\|f\|_{\dot{B}^{-\varrho}_{2,\infty}}}_{Low-frequency\  Estimate}
+\underbrace{(1+t)^{-\frac{\ell}{\sigma_{2}}+\gamma_{\sigma_{2}}(r,2)}\|f\|_{\dot{B}^{s+\ell}_{r,\alpha}}}_{High-frequency\  Estimate}, \label{R-E14}
\end{eqnarray}
for $\ell>n(\frac{1}{r}-\frac{1}{2})$\ \footnote{Let us remark that $\ell\geq0$ in the case of $r=2$.} with $1\leq r\leq2$,
where $\gamma_\sigma(r,p):=\frac{n}{\sigma}(\frac{1}{r}-\frac{1}{p})$.
\end{thm}

%Note that the embedding $L^q(\mathbb{R}^{n})\hookrightarrow\dot{B}^{-\varrho}_{2,\infty}(\mathbb{R}^{n})(\varrho_{q}:=n(1/q-1/2), \ 1\leq q<2)$
%in Lemma, we have the following consequence.

%\begin{cor}
%Let $\eta(\xi)$ be a positive, continuous and real-valued function in $\mathbb{R}^{n}$ satisfying
%\begin{eqnarray}\eta(\xi)\sim\left\{
                % \begin{array}{ll}
                 %  |\xi|^{\sigma_{1}}, &  |\xi|\rightarrow 0; \\
                 %  |\xi|^{-\sigma_{2}}, & |\xi|\rightarrow\infty;
             %    \end{array}
              % \right. \label{R-E166}
%\end{eqnarray}
%for $\sigma_{1}, \sigma_{2}>0$.
%If $f\in \dot{B}^{s+\ell}_{r,\alpha}(\mathbb{R}^{n})\cap L^{q}(\mathbb{R}^{n})$ for $s>-\varrho_{q}$ and $1\leq \alpha\leq\infty$, then it holds that
%\begin{eqnarray}
%&&\Big\|2^{js}\|\widehat{\dot{\Delta}_{j}f}e^{-\eta(\xi)t}\|_{L^{2}}\Big\|_{l^{\alpha}_{j}}\nonumber\\ &\lesssim & \underbrace{(1+t)^{-\gamma_{\sigma_{1}}(q,2)-\frac{s}{\sigma_{1}}}\|f\|_{L^{q}}}_{Low-frequency\  Estimate}
%+\underbrace{(1+t)^{-\frac{\ell}{\sigma_{2}}+\gamma_{\sigma_{2}}(r,2)}\|f\|_{\dot{B}^{s+\ell}_{r,\alpha}}}_{High-frequency\  Estimate},
%\end{eqnarray}
%for $\ell>n(\frac{1}{r}-\frac{1}{2})$\ \footnote{Let us remark that $\ell\geq0$ in the case of $r=2$.} with $1\leq r\leq2$,
%where $\gamma_\sigma(r,p):=\frac{n}{\sigma}(\frac{1}{r}-\frac{1}{p})$.
%\end{cor}
\begin{rem}
Compared to (\ref{R-E12}), more general time-decay estimate (\ref{R-E14}) is endowed with some novelty.  To the best of our knowledge,
not only does the high-frequency part decay in time with algebraic rates of any order as long as the function is spatially regular enough, but also
decay information related to the localized integrability $r$ is available. Indeed, different values of $r$ (for example, $r = 1$ or $r = 2$)
enable us to obtain the desired minimal decay regularity for dissipative systems. Secondly,
note that the embedding $L^p(\mathbb{R}^{n})\hookrightarrow\dot{B}^{-\varrho}_{2,\infty}(\mathbb{R}^{n})(\varrho=n(1/p-1/2), \ 1\leq p<2)$ in Lemma \ref{lem2.3}, it is shown that
the low-frequency regularity is less restrictive than the usual $L^p$ space.
Additionally, the regularity index $s\in \mathbb{R}$ relaxes as the negative real constant rather than non-negative integers in (\ref{R-E12}) only.
\end{rem}

For the convenience of this paper, we give the direct consequence for the dissipative rate of $(1,2)$ type. Obviously, $\sigma_{1}=\sigma_{2}=2$ in this case.
\begin{cor}\label{cor1.1}
Let $\eta(\xi)=|\xi|^2/(1+|\xi|^2)^2$. If $f\in \dot{B}^{s+\ell}_{r,\alpha}(\mathbb{R}^{n})\cap \dot{B}^{-\varrho}_{2,\infty}(\mathbb{R}^{n})$ for $s\in \mathbb{R}, \varrho\in \mathbb{R}$ and $1\leq \alpha\leq\infty$ such that $s+\varrho>0$, then it holds that
\begin{eqnarray}
&&\Big\|2^{qs}\|\widehat{\dot{\Delta}_{q}f}e^{-\eta(\xi)t}\|_{L^{2}}\Big\|_{l^{\alpha}_{q}}\nonumber\\ &\lesssim & \underbrace{(1+t)^{-\frac{s+\varrho}{2}}\|f\|_{\dot{B}^{-\varrho}_{2,\infty}}}_{Low-frequency\  Estimate}
+\underbrace{(1+t)^{-\frac{\ell}{2}+\frac{n}{2}(\frac{1}{r}-\frac{1}{2})}\|f\|_{\dot{B}^{s+\ell}_{r,\alpha}}}_{High-frequency\  Estimate}, \label{R-E15}
\end{eqnarray}
for $\ell>n(\frac{1}{r}-\frac{1}{2})$ with $1\leq r<2$, or $\ell\geq0$ with $r=2$.
\end{cor}

Recalling those results in \cite{X,XXK}, we have the global-in-time existence of classical solutions to  (\ref{R-E1})-(\ref{R-E3}) in spatially critical  Besov spaces.

\begin{thm}\label{thm1.2}
Suppose that the initial data satisfy $w_{0}-w_{\infty}\in B^{5/2}_{2,1}(\mathbb{R}^3)$ and compatible conditions (\ref{R-E4}).
There exists a positive constant $\delta_{0}$ such that if
\begin{equation}
I_{0}:=\|w_{0}-w_{\infty}\|_{B^{5/2}_{2,1}}\leq \delta_{0},\label{R-E16}
\end{equation}
then the system (\ref{R-E1})-(\ref{R-E3}) admits a unique global solution $w(t,x)$ satisfying
\begin{eqnarray*}
w(t,x)\in \mathcal{C}^{1}([0,\infty)\times \mathbb{R}^3)
\end{eqnarray*}
and
\begin{eqnarray*}
w-w_{\infty} \in\widetilde{\mathcal{C}}(B^{5/2}_{2,1}(\mathbb{R}^3))\cap\widetilde{\mathcal{C}}^{1}(B^{3/2}_{2,1}(\mathbb{R}^3)).
\end{eqnarray*}
Moreover, the following energy inequality holds
\begin{equation}
N_{0}(t)+D_{0}(t)\lesssim I_{0},\label{R-E17}
\end{equation}
where $N_{0}(t):=\|w-w_{\infty}\|_{\widetilde{L}^{\infty}_{t}(B^{5/2}_{2,1})}$ and $D_{0}(t):=\|(n-n_{\infty},u)\|_{\widetilde{L}^{2}_{t}(B^{5/2}_{2,1})}
+\|E\|_{\widetilde{L}^{2}_{t}(B^{3/2}_{2,1})}+\|\nabla B\|_{\widetilde{L}^{2}_{t}(B^{1/2}_{2,1})}$ for any $t>0$.
\end{thm}
The reader is referred to Sect.\ref{sec:2} for the space notations. With such low regularity, some decay techniques used in \cite{UK} fail. For instance, Ueda and the second author performed the time-weighted energy estimate related to the norm
$$W^{\bot}(t):=\sup_{0\leq \tau \leq t}(1+\tau)\|(\rho,\upsilon,E)\|_{W^{1,\infty}}$$
to eliminate the difficulty of regularity-loss in the semigroup approach. Consequently, higher regularity was needed to
bound the norm according to Sobolev embedding theorems. In the present paper, we introduce a new strategy to
get around the main obstruction. Based on the Littlewood-Paley pointwise energy estimate (\ref{R-E30}),
we can skip the usual semigroup approach as in \cite{UK,XK2}. Here, localized energy estimates
for (\ref{R-E6})-(\ref{R-E7}) are performed in terms of high-frequency and low-frequency decompositions.
It is worth noting that the time-decay inequality in Corollary \ref{cor1.1} indeed plays the key role to
overcome the weak dissipative mechanism of regularity-loss at the high-frequency. Precisely, we make the best of advantages of (\ref{R-E15}) rather than (\ref{R-E12}). The high-frequency part is divided into two parts, and on each part, different values of $r$ are chosen to obtain desired decay estimates with the assumption regularity $s_{c}=5/2$, see, \textit{e.g.}, (\ref{R-E35}) and (\ref{R-E37}) for details.

Our main result is stated as follows.
\begin{thm}\label{thm1.3}
Assume that the initial data satisfy $w_{0}-w_{\infty}\in B^{5/2}_{2,1}\cap \dot{B}^{-3/2}_{2,\infty}$ and (\ref{R-E4}). Set $I_{1}:=\|w_{0}-w_{\infty}\|_{B^{5/2}_{2,1}\cap \dot{B}^{-3/2}_{2,\infty}}.$
Then there exists a constant $\delta_{1}>0$ such that if $I_{1}\leq\delta_{1}$, then the classical solution to (\ref{R-E1})-(\ref{R-E3}) admits the optimal decay estimate
\begin{eqnarray}
\|w-w_{\infty}\|_{L^2}\lesssim I_{1}(1+t)^{-3/4}. \label{R-E18}
\end{eqnarray}
\end{thm}

In Theorem \ref{thm1.3}, the low frequency regularity assumption is posted in $\dot{B}^{-3/2}_{2,\infty}$, which is less restrictive than the usual $L^1$ space, so a direct consequence follows from Lemma \ref{lem2.3} immediately.

\begin{cor}\label{cor1.2}
Assume that the initial data satisfy $w_{0}-w_{\infty}\in B^{5/2}_{2,1}\cap L^1$.
Set $\widetilde{I}_{1}:=\|w_{0}-w_{\infty}\|_{B^{5/2}_{2,1}\cap L^1}.$
Then there exists a constant $\delta_{1}>0$ such that if $\widetilde{I}_{1}\leq\delta_{1}$, then the classical solution to (\ref{R-E1})-(\ref{R-E3}) satisfies the optimal decay estimate (\ref{R-E18}), where $I_{1}$ is replaced by $\widetilde{I}_{1}$.
\end{cor}

\begin{rem}
In the present paper, we can achieve $s_{D}\leq 5/2$, where the regularity $5/2$ is the same as that
for global solutions in \cite{X,XXK}. This is the first time to show
the optimal decay rate in critical Besov spaces for problem (\ref{R-E1})-(\ref{R-E3}), which allows to reduce the regularity assumption
heavily in comparison with previous works as \cite{D1,TWW,UK} and therein references.
\end{rem}

\begin{rem}
It is worth noting that Theorems \ref{thm1.1}-\ref{thm1.2} encourages us
to investigate the minimal decay regularity for other dissipative systems with regularity-loss as in \cite{Dh,DNK,D2,DS,HK,IHK,IK,LK,SK2}
in the near future.
\end{rem}

Finally, would like to note that all physical parameters are normalized to be one in (\ref{R-E1}). If considered, there are rigorous justifications on the singular-parameter limits for (\ref{R-E1}),  such as the nonrelativistic limit, quasi-neutral limit as well as combined nonrelativistic and
quasi-neutral limits  in  \cite{PW1,PW2,PW3},  diffusive relaxation limits in \cite{XX} and so on.

The rest of this paper unfolds as follows. In Sect.\ref{sec:2}, we
review Littlewood-Paley decomposition, Besov spaces and Chemin-Lerner spaces to make the context possibly self-contained.
 Sect.\ref{sec:3} is devoted to prove the frequency-localization time-decay
inequality. Furthermore, in Sect.\ref{sec:4}, we deduce the optimal decay estimate for (\ref{R-E6})-(\ref{R-E7})
by employing energy approaches in terms of high-frequency and low-frequency decompositions.

\textbf{Notations.}  Throughout the paper, $f\lesssim g$ denotes $f\leq Cg$, where $C>0$
is a generic constant. $f\thickapprox g$ means $f\lesssim g$ and $g\lesssim f$. Denote by $\mathcal{C}([0,T],X)$ (resp.,
$\mathcal{C}^{1}([0,T],X)$) the space of continuous (resp.,
continuously differentiable) functions on $[0,T]$ with values in a
Banach space $X$. Also, $\|(f,g,h)\|_{X}$ means $
\|f\|_{X}+\|g\|_{X}+\|h\|_{X}$, where $f,g,h\in X$.

\section{Preliminary} \setcounter{equation}{0}\label{sec:2}
In this section, we briefly review the Littlewood--Paley
decomposition, Besov spaces and Chemin-Lerner spaces in $\mathbb{R}^{n}(n\geq1)$, see
\cite{BCD} for more details.

Firstly, we give an improved
Bernstein inequality (see \cite{W}), which allows the case of fractional derivatives.

\begin{lem}\label{lem2.1}
Let $0<R_{1}<R_{2}$ and $1\leq a\leq b\leq\infty$.
\begin{itemize}
\item [(i)] If $\mathrm{Supp}\mathcal{F}f\subset \{\xi\in \mathbb{R}^{n}: |\xi|\leq
R_{1}\lambda\}$, then
\begin{eqnarray*}
\|\Lambda^{\alpha}f\|_{L^{b}}
\lesssim \lambda^{\alpha+n(\frac{1}{a}-\frac{1}{b})}\|f\|_{L^{a}}, \ \  \mbox{for any}\ \  \alpha\geq0;
\end{eqnarray*}

\item [(ii)]If $\mathrm{Supp}\mathcal{F}f\subset \{\xi\in \mathbb{R}^{n}:
R_{1}\lambda\leq|\xi|\leq R_{2}\lambda\}$, then
\begin{eqnarray*}
\|\Lambda^{\alpha}f\|_{L^{a}}\approx\lambda^{\alpha}\|f\|_{L^{a}}, \ \  \mbox{for any}\ \ \alpha\in\mathbb{R}.
\end{eqnarray*}
\end{itemize}
\end{lem}

Let ($\varphi, \chi)$ is a
couple of smooth functions valued in [0,1] such that $\varphi$ is
supported in the shell
$\mathcal{C}(0,\frac{3}{4},\frac{8}{3})=\{\xi\in\mathbb{R}^{n}|\frac{3}{4}\leq|\xi|\leq\frac{8}{3}\}$,
$\chi$ is supported in the ball $\mathcal{B}(0,\frac{4}{3})=
\{\xi\in\mathbb{R}^{n}||\xi|\leq\frac{4}{3}\}$ satisfying
$$
\chi(\xi)+\sum_{q\in\mathbb{N}}\varphi(2^{-q}\xi)=1,\ \ \ \ q\in
\mathbb{N},\ \ \xi\in\mathbb{R}^{n}
$$
and
$$
\sum_{k\in\mathbb{Z}}\varphi(2^{-k}\xi)=1,\ \ \ \ k\in \mathbb{Z},\
\ \xi\in\mathbb{R}^{n}\setminus\{0\}.
$$
For $f\in\mathcal{S'}$(the set of temperate distributions
which is the dual of the Schwarz class $\mathcal{S}$),  define
$$
\Delta_{-1}f:=\chi(D)f=\mathcal{F}^{-1}(\chi(\xi)\mathcal{F}f),\
\Delta_{q}f:=0 \ \  \mbox{for}\ \  q\leq-2;
$$
$$
\Delta_{q}f:=\varphi(2^{-q}D)f=\mathcal{F}^{-1}(\varphi(2^{-q}\xi)\mathcal{F}f)\
\  \mbox{for}\ \  q\geq0;
$$
$$
\dot{\Delta}_{q}f:=\varphi(2^{-q}D)f=\mathcal{F}^{-1}(\varphi(2^{-q}\xi)\mathcal{F}f)\
\  \mbox{for}\ \  q\in\mathbb{Z},
$$
where $\mathcal{F}f$, $\mathcal{F}^{-1}f$ represent the Fourier
transform and the inverse Fourier transform on $f$, respectively. Observe that
the operator $\dot{\Delta}_{q}$ coincides with $\Delta_{q}$ for $q\geq0$.

Denote by $\mathcal{S}'_{0}:=\mathcal{S}'/\mathcal{P}$ the tempered distributions modulo polynomials $\mathcal{P}$. Furthermore,
Besov spaces can be characterized by Littlewood-Paley decompositions.

\begin{defn}\label{defn6.1}
For $s\in \mathbb{R}$ and $1\leq p,r\leq\infty,$ the homogeneous
Besov spaces $\dot{B}^{s}_{p,r}$ is defined by
$$\dot{B}^{s}_{p,r}=\{f\in S'_{0}:\|f\|_{\dot{B}^{s}_{p,r}}<\infty\},$$
where
$$\|f\|_{\dot{B}^{s}_{p,r}}
=\left\{\begin{array}{l}\Big(\sum_{q\in\mathbb{Z}}(2^{qs}\|\dot{\Delta}_{q}f\|_{L^p})^{r}\Big)^{1/r},\
\ r<\infty, \\ \sup_{q\in\mathbb{Z}}
2^{qs}\|\dot{\Delta}_{q}f\|_{L^p},\ \ r=\infty.\end{array}\right.
$$\end{defn}

Similarly, one also has the definition of inhomogeneous Besov spaces.
\begin{defn}\label{defn6.2}
For $s\in \mathbb{R}$ and $1\leq p,r\leq\infty,$ the inhomogeneous
Besov spaces $B^{s}_{p,r}$ is defined by
$$B^{s}_{p,r}=\{f\in S':\|f\|_{B^{s}_{p,r}}<\infty\},$$
where
$$\|f\|_{B^{s}_{p,r}}
=\left\{\begin{array}{l}\Big(\sum_{q=-1}^{\infty}(2^{qs}\|\Delta_{q}f\|_{L^p})^{r}\Big)^{1/r},\
\ r<\infty, \\ \sup_{q\geq-1} 2^{qs}\|\Delta_{q}f\|_{L^p},\ \
r=\infty.\end{array}\right.
$$
\end{defn}

Besov spaces obey various inclusion relations. Precisely,
\begin{lem}\label{lem2.2} Let $s\in \mathbb{R}$ and $1\leq
p,r\leq\infty,$ then
\begin{itemize}
\item [(i)]If $s>0$, then $B^{s}_{p,r}=L^{p}\cap \dot{B}^{s}_{p,r};$
\item[(ii)]If $\tilde{s}\leq s$, then $B^{s}_{p,r}\hookrightarrow
B^{\tilde{s}}_{p,r}$. This inclusion relation is false for
the homogeneous Besov spaces;
\item[(iii)]If $1\leq r\leq\tilde{r}\leq\infty$, then $\dot{B}^{s}_{p,r}\hookrightarrow
\dot{B}^{s}_{p,\tilde{r}}$ and $B^{s}_{p,r}\hookrightarrow
B^{s}_{p,\tilde{r}};$
\item[(iv)]If $1\leq p\leq\tilde{p}\leq\infty$, then $\dot{B}^{s}_{p,r}\hookrightarrow \dot{B}^{s-n(\frac{1}{p}-\frac{1}{\tilde{p}})}_{\tilde{p},r}
$ and $B^{s}_{p,r}\hookrightarrow
B^{s-n(\frac{1}{p}-\frac{1}{\tilde{p}})}_{\tilde{p},r}$;
\item[(v)]$\dot{B}^{n/p}_{p,1}\hookrightarrow\mathcal{C}_{0},\ \ B^{n/p}_{p,1}\hookrightarrow\mathcal{C}_{0}(1\leq p<\infty);$
\end{itemize}
where $\mathcal{C}_{0}$ is the space of continuous bounded functions
which decay at infinity.
\end{lem}

In the recent work \cite{SS},  Sohinger and Strain first introduced the
Besov space of negative order to investigate the optimal time-decay rate of Boltzmann equation. Here, we
give the simplification as follows.
\begin{lem}\label{lem2.3}
Suppose that $\varrho>0$ and $1\leq p<2$. It holds that
\begin{eqnarray*}
\|f\|_{\dot{B}^{-\varrho}_{r,\infty}}\lesssim \|f\|_{L^{p}}
\end{eqnarray*}
with $1/p-1/r=\varrho/n$. In particular, this holds with $\varrho=n/2, r=2$ and $p=1$.
\end{lem}

Usually, Moser-type product estimates play an important role in the estimate of bilinear
terms.
\begin{prop}\label{prop2.1}
Let $s>0$ and $1\leq
p,r\leq\infty$. Then $\dot{B}^{s}_{p,r}\cap L^{\infty}$ is an algebra and
$$
\|fg\|_{\dot{B}^{s}_{p,r}}\lesssim \|f\|_{L^{\infty}}\|g\|_{\dot{B}^{s}_{p,r}}+\|g\|_{L^{\infty}}\|f\|_{\dot{B}^{s}_{p,r}}.
$$
Let $s_{1},s_{2}\leq n/p$ such that $s_{1}+s_{2}>n\max\{0,\frac{2}{p}-1\}. $  Then one has
$$\|fg\|_{\dot{B}^{s_{1}+s_{2}-n/p}_{p,1}}\lesssim \|f\|_{\dot{B}^{s_{1}}_{p,1}}\|g\|_{\dot{B}^{s_{2}}_{p,1}}.$$
\end{prop}

In the analysis of decay estimates, we use the general form of Moser-type product estimates, which was shown by \cite{Z}.
\begin{prop}\label{prop2.2}
Let $s>0$ and $1\leq p,r,p_{1},p_{2},p_{3},p_{4}\leq\infty$. Assume
that $f\in L^{p_{1}}\cap \dot{B}^{s}_{p_{4},r}$ and $g\in L^{p_{3}}\cap
\dot{B}^{s}_{p_{2},r}$ with
$$\frac{1}{p}=\frac{1}{p_{1}}+\frac{1}{p_{2}}=\frac{1}{p_{3}}+\frac{1}{p_{4}}.$$
Then it holds that
\begin{eqnarray*}
\|fg\|_{\dot{B}^{s}_{p,r}}\lesssim \|f\|_{L^{p_{1}}}\|g\|_{\dot{B}^{s}_{p_{2},r}}+\|g\|_{L^{p_{3}}}\|f\|_{\dot{B}^{s}_{p_{4},r}}.
\end{eqnarray*}
\end{prop}

On the other hand, we need space-time mixed Besov spaces initiated by J.-Y. Chemin and N. Lerner
in \cite{CL}, which can be regarded as the refinement of the usual space-time mixed spaces
$L^{\theta}_{T}(\dot{B}^{s}_{p,r})$ or $L^{\theta}_{T}(B^{s}_{p,r})$.
\begin{defn}\label{defn6.3}
For $T>0, s\in\mathbb{R}, 1\leq r,\theta\leq\infty$, the homogeneous
Chemin-Lerner spaces
$\widetilde{L}^{\theta}_{T}(\dot{B}^{s}_{p,r})$ is defined by
$$\widetilde{L}^{\theta}_{T}(\dot{B}^{s}_{p,r}):
=\{f\in
L^{\theta}(0,T;\mathcal{S}'_{0}):\|f\|_{\widetilde{L}^{\theta}_{T}(\dot{B}^{s}_{p,r})}<+\infty\},$$
where
$$\|f\|_{\widetilde{L}^{\theta}_{T}(\dot{B}^{s}_{p,r})}:=\Big(\sum_{q\in\mathbb{Z}}(2^{qs}\|\dot{\Delta}_{q}f\|_{L^{\theta}_{T}(L^{p})})^{r}\Big)^{\frac{1}{r}}$$
with the usual convention if $r=\infty$.
\end{defn}

\begin{defn}\label{defn6.4}
For $T>0, s\in\mathbb{R}, 1\leq r,\theta\leq\infty$, the
inhomogeneous Chemin-Lerner spaces
$\widetilde{L}^{\theta}_{T}(B^{s}_{p,r})$ is defined by
$$\widetilde{L}^{\theta}_{T}(B^{s}_{p,r}):
=\{f\in
L^{\theta}(0,T;\mathcal{S}'):\|f\|_{\widetilde{L}^{\theta}_{T}(B^{s}_{p,r})}<+\infty\},$$
where
$$\|f\|_{\widetilde{L}^{\theta}_{T}(B^{s}_{p,r})}:=\Big(\sum_{q\geq-1}(2^{qs}\|\Delta_{q}f\|_{L^{\theta}_{T}(L^{p})})^{r}\Big)^{\frac{1}{r}}$$
with the usual convention if $r=\infty$.
\end{defn}

We further define
$$\widetilde{\mathcal{C}}_{T}(B^{s}_{p,r}):=\widetilde{L}^{\infty}_{T}(B^{s}_{p,r})\cap\mathcal{C}([0,T],B^{s}_{p,r})
$$ and $$\widetilde{\mathcal{C}}^1_{T}(B^{s}_{p,r}):=\{f\in\mathcal{C}^1([0,T],B^{s}_{p,r})|\partial_{t}f\in\widetilde{L}^{\infty}_{T}(B^{s}_{p,r})\},$$
where the index $T$ will be omitted when $T=+\infty$.

By Minkowski's inequality, Chemin-Lerner spaces can be linked with
$L^{\theta}_{T}(X)$ with $X=B^{s}_{p,r}$ or $\dot{B}^{s}_{p,r}$.
\begin{rem}\label{rem6.1}
It holds that
$$\|f\|_{\widetilde{L}^{\theta}_{T}(X)}\leq\|f\|_{L^{\theta}_{T}(X)}\,\,\,
\mbox{if}\,\, r\geq\theta;\ \ \ \
\|f\|_{\widetilde{L}^{\theta}_{T}(X)}\geq\|f\|_{L^{\theta}_{T}(X)}\,\,\,
\mbox{if}\,\, r\leq\theta.
$$\end{rem}

\section{The proof of Theorem \ref{thm1.1}} \setcounter{equation}{0}\label{sec:3}
In the recent decade, harmonic analysis tools, especially for techniques based on Littlewood-Paley decomposition and paradifferential calculus
have proved to be very efficient in the study of partial differential equations, see for example, \cite{BCD}. It is well-known that the frequency-localization operator $\dot{\Delta}_{q}f$ (or $\Delta_{q}f$ ) has a smoothing effect on the function $f$, even though $f$ is quite rough. Moreover, the $L^p$-norm of $\dot{\Delta}_{q}f$ can be preserved if $f\in L^p(\mathbb{R}^{n})$. However, so far there are few efforts on the time-decay property related to the block operator, so
Theorem \ref{thm1.1} seems to be a suitable candidate for the motivation, which enables us to overcome the outstanding difficulty of regularity-loss
in Besov spaces with relatively lower regularity.

\begin{proof}
Indeed, we proceed the proof for the inequality (\ref{R-E14}) with the aid of Littlewood-Paley frequency-localization techniques. It follows from the assumption (\ref{R-E13}) that there exist constants $c_{0}>0$ and $R_{0}>0$ such that
\begin{eqnarray}
&&\|\widehat{\dot{\Delta}_{q}f}e^{-\eta(\xi)t}\|_{L^{2}}\nonumber\\
&\leq& \|\widehat{\dot{\Delta}_{q}f}e^{-c_{0}|\xi|^{\sigma_{1}}t}\|_{L^{2}(|\xi|\leq R_{0})}+
\|\widehat{\dot{\Delta}_{q}f}e^{-c_{0}|\xi|^{-\sigma_{2}}t}\|_{L^{2}(|\xi|\geq R_{0})}.
 \label{R-E19}
\end{eqnarray}
We set $R_{0}=2^{q_{0}}(q_{0}\in \mathbb{Z})$ without the loss of generality.

(1) If $q\geq q_{0}$, then $|\xi|\sim2^{q}\geq R_{0}$, which leads to
\begin{eqnarray}
&&\|\widehat{\dot{\Delta}_{q}f}e^{-c_{0}|\xi|^{-\sigma_{2}}t}\|_{L^{2}(|\xi|\geq R_{0})}
\nonumber\\&=&
\Big\||\xi|^{\ell}|\widehat{\dot{\Delta}_{q}f}|\frac{e^{-c_{0}t|\xi|^{-\sigma_{2}}}}{|\xi|^{\ell}}\Big\|_{L^2(|\xi|\geq R_{0})}\nonumber\\&\leq&
\||\xi|^{\ell}\widehat{\dot{\Delta}_{q}f}\|_{L^{r'}}\Big\|\frac{e^{-c_{0}t|\xi|^{-\sigma_2}}}{|\xi|^{\ell}}\Big\|_{L^{m}(|\xi|\geq R_{0})}
 \ \ \Big(\frac{1}{r'}+\frac{1}{m}=\frac{1}{2},\ r'\geq2 \Big)\nonumber\\&\leq&
 2^{q\ell}\|\dot{\Delta}_{q}f\|_{L^{r}}\Big\|\frac{e^{-c_{0}t|\xi|^{-\sigma_2}}}{|\xi|^{\ell}}\Big\|_{L^{m}(|\xi|\geq R_{0})}\ \  \Big(\frac{1}{r}+\frac{1}{r'}=1\Big), \label{R-E20}
\end{eqnarray}
where the Hausdorff-Young's inequality was used in the last line. By performing the change of variable, we can arrive at
\begin{eqnarray}
\Big\|\frac{e^{-c_{0}t|\xi|^{-\sigma_2}}}{|\xi|^{\ell}}\Big\|_{L^{m}(|\xi|\geq R_{0})}\lesssim (1+t)^{-\frac{\ell}{\sigma_2}+\frac{n}{\sigma_2}(\frac{1}{r}-\frac{1}{2})} \label{R-E21}
\end{eqnarray}
for $\ell>n(\frac{1}{r}-\frac{1}{2})$. Besides, it can be also bounded by $(1+t)^{-\frac{\ell}{\sigma_2}}$
for $\ell\geq0$ if $r=2$.
Then it follows from (\ref{R-E20})-(\ref{R-E21})  that
\begin{eqnarray}
2^{qs}\|\widehat{\dot{\Delta}_{q}f}e^{-c_{0}|\xi|^{-\sigma_{2}}t}\|_{L^{2}}\lesssim 2^{q(s+\ell)}(1+t)^{-\frac{\ell}{\sigma_2}+\frac{n}{\sigma_2}(\frac{1}{r}-\frac{1}{2})}\|\dot{\Delta}_{q}f\|_{L^{r}}. \label{R-E22}
\end{eqnarray}

(2) If $q<q_{0}$, then $|\xi|\sim2^{q}\leq R_{0}$, which implies that
\begin{eqnarray}
\|\widehat{\dot{\Delta}_{q}f}e^{-c_{0}|\xi|^{\sigma_{1}}t}\|_{L^{2}(|\xi|\leq R_{0})}
\lesssim \|\widehat{\dot{\Delta}_{q}f}\|_{L^{2}}e^{-c_{0}(2^{q}\sqrt[\sigma_{1}]{t})^{\sigma_{1}}}. \label{R-E23}
\end{eqnarray}
 Furthermore, we can obtain
\begin{eqnarray}
2^{qs}\|\widehat{\dot{\Delta}_{q}f}e^{-c_{0}|\xi|^{\sigma_{1}}t}\|_{L^{2}}
\lesssim\|f\|_{\dot{B}^{-\varrho}_{2,\infty}}(1+t)^{-\frac{s+\varrho}{\sigma_{1}}}[(2^{q}\sqrt[\sigma_{1}]{t})^{s+\varrho}e^{-c_{0}(2^{q}\sqrt[\sigma_{1}]{t})^{\sigma_{1}}}] \label{R-E24}
\end{eqnarray}
for $s\in \mathbb{R}, \varrho\in \mathbb{R}$ such that $s+\varrho>0$. Note that
\begin{eqnarray}
\Big\|(2^{q}\sqrt[\sigma_{1}]{t})^{s+\varrho}e^{-c_{0}(2^{q}\sqrt[\sigma_{1}]{t})^{\sigma_{1}}}\Big\|_{l^{\alpha}_{q}}\lesssim 1, \label{R-E25}
\end{eqnarray}
for any $\alpha\in [1,+\infty]$. Combining (\ref{R-E22}), (\ref{R-E24})-(\ref{R-E25}), we conclude that
\begin{eqnarray}
&&\Big\|2^{qs}\|\widehat{\dot{\Delta}_{q}f}e^{-\eta(\xi)t}\|_{L^{2}}\Big\|_{l^{\alpha}_{q}}\nonumber\\ &\lesssim & \|f\|_{\dot{B}^{-\varrho}_{2,\infty}}(1+t)^{-\frac{s+\varrho}{\sigma_{1}}}
+\|f\|_{\dot{B}^{s+\ell}_{r,\alpha}}(1+t)^{-\frac{\ell}{\sigma_2}+\frac{n}{\sigma_2}(\frac{1}{r}-\frac{1}{2})}, \label{R-E26}
\end{eqnarray}
which is just the inequality (\ref{R-E14}).
\end{proof}

\section{The proof of Theorem \ref{thm1.2}} \setcounter{equation}{0}\label{sec:4}
Due to the dissipative mechanism of regularity-loss, extra higher regularity is usually needed to obtain the optimal decay rate for (\ref{R-E1})-(\ref{R-E3}).
To achieve the minimal decay regularity $s_{D}\leq5/2$,  we shall skip the usual
semigroup approach as in \cite{UK,XK2}. Consequently, the nonlinear energy estimate in Fourier spaces for (\ref{R-E6})-(\ref{R-E7}) is performed. We would like to mention that similar estimates were first given by the second author in \cite{Ka2} for the Boltzmann equation, then well developed in \cite{KY2} for hyperbolic systems of balance laws. In the following, our decay analysis focus on (\ref{R-E6})-(\ref{R-E7}). As a matter of fact, as Theorem \ref{thm1.2}, we can obtain a similar global existence for the solution $z$ to (\ref{R-E6})-(\ref{R-E7}).
For simplification, allow us to abuse the notations $N_{0}(t)$ and $D_{0}(t)$ a little, which means that the corresponding functional norms with respect to $z$  is still labelled as $N_{0}(t)$ and $D_{0}(t)$. Then it follows that these norms can be bounded by
$\|z_0\|_{B^{5/2}_{2,1}}$.

Define
\begin{equation}
N(t)=\sup_{0\leq\tau \leq t}(1+\tau)^{\frac{3}{4}}\|z(\tau)\|_{L^{2}},\nonumber
\end{equation}
\begin{equation}
D(t)=\|(\rho,\upsilon)\|_{L^{2}_{t}(B^{5/2}_{2,1})}+\|E\|_{L^{2}_{t}(B^{3/2}_{2,1})}+
\|\nabla h\|_{L^{2}_{t}(B^{1/2}_{2,1})}.\nonumber
\end{equation}
The optimal decay estimate lies in a nonlinear time-weighted energy inequality, which is included in the following
\begin{lem}\label{lem4.1}
Let $z=(\rho,\upsilon,E,h)^{\top}$ be the global classical solutions, which is similar to that in Theorem \ref{thm1.2}. Additionally, if $z_{0}\in\dot{B}^{-3/2}_{2,\infty}$, then it holds that
\begin{eqnarray}
N(t)\lesssim \|z_{0}\|_{B^{5/2}_{2,1}\cap \dot{B}^{-3/2}_{2,\infty}}+N(t)D(t)+N(t)^2. \label{R-E27}
\end{eqnarray}
\end{lem}

\begin{proof}
It follows from the nonlinear energy method in \cite{XMK} that
\begin{eqnarray}
\frac{d}{dt}\mathcal{E}[\hat{z}]+c_{1}\eta_{0}(\xi)\mathcal{E}[\hat{z}]\lesssim (|\xi|^2|\hat{Q}|^2+|\hat{R}|^2), \label{R-E28}
\end{eqnarray}
for $c_{1}>0$, where $\eta_{0}(\xi)=|\xi|^2/(1+|\xi|^2)^2$ and $\mathcal{E}[\hat{z}]\approx |\hat{z}|^2$. As a matter of fact,
the corresponding Littlewood-Paley pointwise energy estimate is also available according to the derivation of (\ref{R-E28}):
\begin{eqnarray}
\frac{d}{dt}\mathcal{E}[\widehat{\dot{\Delta}_{q}z}]+c_{1}\eta_{0}(\xi)\mathcal{E}[\widehat{\dot{\Delta}_{q}z}]\lesssim (|\xi|^2|\widehat{\dot{\Delta}_{q}Q}|^2+|\widehat{\dot{\Delta}_{q}R}|^2), \label{R-E29}
\end{eqnarray}
where $\mathcal{E}[\widehat{\dot{\Delta}_{q}z}]\approx |\widehat{\dot{\Delta}_{q}z}|^2$. The standard Gronwall's inequality implies that
\begin{eqnarray}
|\widehat{\dot{\Delta}_{q}z}|^2\lesssim e^{-c_{1}\eta_{0}(\xi)t}|\widehat{\dot{\Delta}_{q}z_{0}}|^2+\int_{0}^{t}e^{-c_{1}\eta_{0}(\xi)(t-\tau)}\Big(|\xi|^2|\widehat{\dot{\Delta}_{q}Q}|^2+|\widehat{\dot{\Delta}_{q}R}|^2\Big)d\tau.
\label{R-E30}
\end{eqnarray}
It follows from Fubini and Plancherel theorems that
\begin{eqnarray}
\|z\|^{2}_{L^{2}}&=&\sum_{q\in \mathbb{Z}}\|\dot{\Delta}_{q}z\|^{2}_{L^{2}}
\nonumber\\&\lesssim& \sum_{q\in \mathbb{Z}} \|\widehat{\dot{\Delta}_{q}z_{0}}e^{-\frac{1}{2}c_{1}\eta_{0}(\xi)t}\|^{2}_{L^{2}}
\nonumber\\&&+\int^{t}_{0} \sum_{q\in \mathbb{Z}}\Big(\||\xi|\widehat{\dot{\Delta}_{q}Q}e^{-\frac{1}{2}c_{1}\eta_{0}(\xi)(t-\tau)}\|^2_{L^{2}}+\|\widehat{\dot{\Delta}_{q}R}e^{-\frac{1}{2}c_{1}\eta_{0}(\xi)(t-\tau)}\|^2_{L^{2}}\Big)d\tau
\nonumber\\&& \triangleq J_{1}+J_{2}+J_{3}.\label{R-E31}
\end{eqnarray}
For $J_{1}$, by taking $r=\alpha=2, s=0, \varrho=3/2$ and $\ell=2$ in Corollary \ref{cor1.1}, we arrive at
\begin{eqnarray}
J_{1}&=&\Big(\sum_{q<q_{0}}+\sum_{q\geq q_{0}}\Big)\Big(\cdot\cdot\cdot\Big)\nonumber
\\&\lesssim& \|z_{0}\|^{2}_{\dot{B}^{-3/2}_{2,\infty}}(1+t)^{-\frac{3}{2}}+\sum_{q\geq q_{0}}2^{2q}\|\dot{\Delta}_{q}z_{0}\|^{2}_{L^{2}}(1+t)^{-2}\nonumber
\\&\lesssim&\|z_{0}\|^{2}_{\dot{B}^{-3/2}_{2,\infty}}(1+t)^{-\frac{3}{2}}+\|z_{0}\|^2_{\dot{B}^{2}_{2,2}}(1+t)^{-2}\nonumber
\\&\lesssim& \|z_{0}\|^{2}_{\dot{B}^{-3/2}_{2,\infty}\cap B^{5/2}_{2,1}}(1+t)^{-\frac{3}{2}}, \label{R-E32}
\end{eqnarray}
where we have used the embedding relation in Lemma \ref{lem2.2}.
Next, we begin to bound nonlinear terms on the right-hand side of (\ref{R-E31}). For $J_{2}$, it can be written as the sum of low-frequency and high-frequency:
\begin{eqnarray}
J_{2}=\int^{t}_{0}\Big(\sum_{q<q_{0}}+\sum_{q\geq q_{0}}\Big)\Big(\cdot\cdot\cdot\Big)\triangleq J_{2L}+J_{2H}. \label{R-E33}
\end{eqnarray}
For $J_{2L}$, by taking $\alpha=2, s=1, \varrho=3/2$ in Corollary \ref{cor1.1}, we have
\begin{eqnarray}
J_{2L}&\leq&\int^{t}_{0}(1+t-\tau)^{-\frac{5}{2}}\|Q(\tau)\|^{2}_{\dot{B}^{-3/2}_{2,\infty}}d\tau\nonumber\\&
\lesssim&\int^{t}_{0}(1+t-\tau)^{-\frac{5}{2}}\|Q(\tau)\|^{2}_{L^{1}}d\tau
\nonumber\\&
\lesssim& \int^{t}_{0}(1+t-\tau)^{-\frac{5}{2}}\|z^{\bot}(\tau)\|^4_{L^{2}}d\tau
\nonumber\\&\lesssim& N^{4}(t)\int^{t}_{0}(1+t-\tau)^{-\frac{5}{2}}(1+\tau)^{-3}d\tau
\nonumber\\&\lesssim& N^{4}(t)(1+t)^{-\frac{5}{2}}, \label{R-E34}
\end{eqnarray}
where we have used the embedding $L^1(\mathbb{R})\hookrightarrow \dot{B}^{-3/2}_{2,\infty}(\mathbb{R}^3)$ in Lemma \ref{lem2.3} and
the fact $Q(z)=O(|z^{\bot}|^2)$ with $z^{\bot}:=(\rho,\upsilon)$.

For the high-frequency part $J_{2H}$, we need more elaborate decay analysis to achieve the aim of $s_{D}\leq5/2$.
To do this, we write
\begin{eqnarray}
J_{2H}=\Big(\int^{t/2}_{0}+\int^{t}_{t/2}\Big)\Big(\cdot\cdot\cdot\Big)\triangleq J_{2H1}+J_{2H2}.\nonumber
\end{eqnarray}
For $J_{2H1}$, taking $r=\alpha=2, s=1$ and $\ell=3/2$ in Corollary \ref{cor1.1} gives
\begin{eqnarray}
J_{2H1}&&=\int^{t/2}_{0}\sum_{q\geq q_{0}}\||\xi|\widehat{\dot{\Delta}_{q}Q}e^{-\frac{1}{2}c_{1}\eta_{0}(\xi)(t-\tau)}\|^2_{L^{2}}d\tau\nonumber
\\&&\leq\int_{0}^{t/2}(1+t-\tau)^{-\frac{3}{2}}\|Q(\tau)\|^2_{\dot{B}_{2,2}^{5/2}}d\tau\nonumber
\\&&\leq\int_{0}^{t/2}(1+t-\tau)^{-\frac{3}{2}}\|z^{\bot}\|^{2}_{L^{\infty}}\|z^{\bot}\|^{2}_{B_{2,1}^{5/2}}d\tau\nonumber
\\&&\leq\sup_{0\leq\tau\leq t/2}\{(1+t-\tau)^{-\frac{3}{2}}\|z\|^{2}_{L^{\infty}}\}\int_{0}^{t/2}\|(\rho,\upsilon)\|^{2}_{B^{5/2}_{2,1}}d\tau\nonumber
\\&&\lesssim(1+t)^{-\frac{3}{2}}\|z_{0}\|^{2}_{B^{5/2}_{2,1}}D^{2}(t)\nonumber
\\&&\lesssim(1+t)^{-\frac{3}{2}}\|z_{0}\|^{4}_{B^{5/2}_{2,1}}, \label{R-E35}
\end{eqnarray}
where we have used Proposition \ref{prop2.1} and Lemma \ref{lem2.2}. Additionally, we would like to explain a little for the last step of (\ref{R-E35}) . It follows from  Remark \ref{rem6.1} that
\begin{eqnarray}
D(t)\lesssim D_{0}(t) \lesssim \|z_{0}\|_{B^{5/2}_{2,1}},\label{R-E36}
\end{eqnarray}
where the energy inequality (\ref{R-E17}) in Theorem \ref{thm1.2} was used.

By choosing $\alpha=2, r=s=1$ and $\ell=3/2$ in Corollary \ref{cor1.1}, $I_{2H2}$ is proceeded as
\begin{eqnarray}
J_{2H2}&=&\int^{t}_{t/2}\sum_{q\geq q_{0}}\||\xi|\widehat{\dot{\Delta}_{q}Q}e^{-\frac{1}{2}c_{1}\eta_{0}(\xi)(t-\tau)}\|^2_{L^{2}}d\tau\nonumber
\\&\leq&\int^{t}_{t/2}\sum_{q\geq q_{0}}2^{2q(\frac{3}{2}+1)}\|\dot{\Delta}_{q}Q\|^{2}_{L^{1}}(1+t-\tau)^{-\frac{3}{2}+3(1-\frac{1}{2})}d\tau\nonumber
\\&=&\int^{t}_{t/2}\|Q(\tau)\|^{2}_{\dot{B}^{5/2}_{1,2}}d\tau. \label{R-E37}
\end{eqnarray}
Recalling $Q(z)=O(|z^{\bot}|^2)$, it follows from Proposition \ref{prop2.2} that
\begin{eqnarray}
\|Q(z)\|_{\dot{B}^{5/2}_{1,2}}\leq\|Q(z)\|_{\dot{B}^{5/2}_{1,1}}\lesssim\|z^{\bot}\|_{L^{2}}\|z^{\bot}\|_{\dot{B}^{5/2}_{2,1}}. \label{R-E38}
\end{eqnarray}
Together with (\ref{R-E37})-(\ref{R-E38}), we are led to
\begin{eqnarray}
J_{2H2}&\leq&\int^{t}_{t/2}\|z^{\bot}(\tau)\|^2_{L^{2}}\|z^{\bot}(\tau)\|^2_{\dot{B}^{5/2}_{2,1}}d\tau\nonumber
\\&\lesssim& N(t)^{2}\int^{t}_{t/2}(1+\tau)^{-\frac{3}{2}}\|z^{\bot}(\tau)\|^2_{\dot{B}^{5/2}_{2,1}}d\tau\nonumber
\\&\lesssim& N(t)^{2}\sup_{t/2\leq\tau\leq t}\Big\{(1+\tau)^{-\frac{3}{2}}\Big\}\int^{t}_{t/2}\|(\rho,\upsilon)\|^{2}_{B^{5/2}_{2,1}}d\tau\nonumber
\\&\lesssim& (1+t)^{-\frac{3}{2}}N(t)^{2}D(t)^{2}. \label{R-E40}
\end{eqnarray}
Hence, combine (\ref{R-E35}) and (\ref{R-E40}) to get
\begin{eqnarray}
J_{2H}\lesssim(1+t)^{-\frac{3}{2}}\|z_{0}\|^{4}_{B^{5/2}_{2,1}}+(1+t)^{-\frac{3}{2}}N(t)^{2}D(t)^{2}. \label{R-E41}
\end{eqnarray}
Furthermore, it follows from (\ref{R-E34}) and (\ref{R-E41}) that
\begin{eqnarray}
J_{2}\lesssim(1+t)^{-\frac{3}{2}}\|z_{0}\|^{4}_{B^{5/2}_{2,1}}+(1+t)^{-\frac{3}{2}}N(t)^{2}D(t)^{2}+N(t)^{4}(1+t)^{-\frac{5}{2}}. \label{R-E42}
\end{eqnarray}

For $J_{3}$, we can perform the similar decay estimates as $J_{2}$. Firstly, we write
\begin{eqnarray}
J_{3}=\int^{t}_{0}\Big(\sum_{q<q_{0}}+\sum_{q\geq q_{0}}\Big)\Big(\cdot\cdot\cdot\Big)\triangleq J_{3L}+J_{3H}. \label{R-E43}
\end{eqnarray}
Note that $R(z)=O(\rho |E|+|\upsilon||h|)$, by taking $\alpha=2, s=0$ and $ \varrho=3/2$ in Corollary \ref{cor1.1},
we obtain
\begin{eqnarray}
J_{3L}&=&\int^{t}_{0}(1+t-\tau)^{-\frac{3}{2}}\|R(\tau)\|^{2}_{\dot{B}^{-3/2}_{2,\infty}}d\tau\nonumber
\\&\leq&\int^{t}_{0}(1+t-\tau)^{-\frac{3}{2}}\|R(\tau)\|^{2}_{L^{1}}d\tau\nonumber
\\&\leq&\int^{t}_{0}(1+t-\tau)^{-\frac{3}{2}}\|z(\tau)\|^{4}_{L^{2}}d\tau\nonumber
\\&\lesssim& N(t)^{4}\int^{t}_{0}(1+t-\tau)^{-\frac{3}{2}}(1+\tau)^{-3}d\tau\nonumber
\\&=&N(t)^{4}(1+t)^{-\frac{3}{2}}, \label{R-E44}
\end{eqnarray}
where Lemma \ref{lem2.3} was well used.
For the high-frequency part, we separate it into two parts
\begin{eqnarray}
J_{3H}=\Big(\int^{t/2}_{0}+\int^{t}_{t/2}\Big)\Big(\cdot\cdot\cdot\Big)\triangleq J_{3H1}+J_{3H2}.\nonumber
\end{eqnarray}
For $J_{3H1}$, taking $r=\alpha=2, s=0$ and $\ell=3/2$ in Corollary \ref{cor1.1} gives
\begin{eqnarray}
J_{3H1}&\lesssim &\int^{t/2}_{0}(1+t-\tau)^{-\frac{3}{2}}\|R(\tau)\|^{2}_{\dot{B}^{3/2}_{2,2}}d\tau. \label{R-E441}
\end{eqnarray}
It follows from Lemma \ref{lem2.1} and Proposition \ref{prop2.1} that
\begin{eqnarray}
\|R\|_{\dot{B}^{3/2}_{2,2}}&\leq&\|R\|_{\dot{B}^{3/2}_{2,1}}\nonumber
\\&\lesssim &\|(\varrho,\upsilon)\|_{L^\infty}\Big(\|E\|_{\dot{B}^{3/2}_{2,1}}+\|\nabla h\|_{\dot{B}^{1/2}_{2,1}}\Big)\nonumber
\\&& +\|(h,E)\|_{L^\infty}\|(\varrho,\upsilon)\|_{\dot{B}^{3/2}_{2,1}}\nonumber
\\&\lesssim& \|z\|_{L^\infty}\Big(\|(\varrho,\upsilon,E)\|_{\dot{B}^{3/2}_{2,1}}+\|\nabla h\|_{\dot{B}^{1/2}_{2,1}}\Big). \label{R-E442}
\end{eqnarray}
Therefore, substituting (\ref{R-E442}) into (\ref{R-E441}) gives
\begin{eqnarray}
J_{3H1}&\lesssim&\int^{t/2}_{0}(1+t-\tau)^{-\frac{3}{2}}\|z(\tau)\|^{2}_{L^{\infty}}\Big(\|(\varrho,\upsilon,E)\|_{\dot{B}^{3/2}_{2,1}}+\|\nabla h\|_{\dot{B}^{1/2}_{2,1}}\Big)^2d\tau\nonumber
\\&\lesssim&\sup_{0\leq\tau\leq t/2}\Big\{(1+t-\tau)^{-\frac{3}{2}}\|z(\tau)\|^{2}_{L^{\infty}}\Big\}\nonumber\\ &&\times \int_{0}^{t/2}\Big(\|(\varrho,\upsilon)\|^2_{B^{5/2}_{2,1}}+\|E\|^2_{B^{3/2}_{2,1}}+\|\nabla h\|^2_{B^{1/2}_{2,1}}\Big)d\tau\nonumber
\\&\lesssim&(1+t)^{-\frac{3}{2}}\|z_{0}\|^{2}_{B^{5/2}_{2,1}}D(t)^{2}\nonumber
\\&\lesssim&(1+t)^{-\frac{3}{2}}\|z_{0}\|^{2}_{B^{5/2}_{2,1}}, \label{R-E45}
\end{eqnarray}
where we have used Lemma \ref{lem2.2} and  (\ref{R-E36}).

On the other hand, by taking $\alpha=2, r=1, s=0$ and $\ell=3/2$ in Corollary \ref{cor1.1}, we arrive at
\begin{eqnarray}
J_{3H2}&=&\int^{t}_{t/2}\sum_{q\geq q_{0}}2^{3q}\|\dot{\Delta}_{q}R(\tau)\|^{2}_{L^{1}}d\tau\nonumber
\\&\lesssim&\int^{t}_{t/2}\|R(\tau)\|^{2}_{\dot{B}^{3/2}_{1,2}}d\tau, \label{R-E46}
\end{eqnarray}
where Lemma \ref{lem2.1}, Lemma \ref{lem2.2}  and Proposition \ref{prop2.2} enable us to obtain
\begin{eqnarray}
\|R\|_{\dot{B}^{3/2}_{1,2}}&\leq&\|R\|_{\dot{B}^{3/2}_{1,1}}\nonumber
\\&\lesssim &\|(\varrho,\upsilon)\|_{L^2}\Big(\|E\|_{\dot{B}^{3/2}_{2,1}}+\|\nabla h\|_{\dot{B}^{1/2}_{2,1}}\Big)\nonumber
\\&& +\|(h,E)\|_{L^2}\|(\varrho,\upsilon)\|_{\dot{B}^{3/2}_{2,1}}\nonumber
\\&\lesssim& \|z\|_{L^2}\Big(\|(\varrho,\upsilon,E)\|_{\dot{B}^{3/2}_{2,1}}+\|\nabla h\|_{\dot{B}^{1/2}_{2,1}}). \label{R-E47}
\end{eqnarray}
Together with (\ref{R-E46})-(\ref{R-E47}), we are led to
\begin{eqnarray}
J_{3H2}&\lesssim&\int_{t/2}^{t}\|z(\tau)\|^2_{L^{2}}\Big(\|(\varrho,\upsilon,E)\|_{\dot{B}^{3/2}_{2,1}}+\|\nabla h\|_{\dot{B}^{1/2}_{2,1}}\Big)^2d\tau\nonumber
\\&\lesssim& N(t)^{2}\int_{t/2}^{t}(1+\tau)^{-\frac{3}{2}}\Big(\|(\varrho,\upsilon)\|^2_{B^{5/2}_{2,1}}+\|E\|^2_{B^{3/2}_{2,1}}+\|\nabla h\|^2_{B^{1/2}_{2,1}}\Big)d\tau\nonumber
\\&\lesssim& N(t)^{2}\sup_{t/2\leq\tau\leq t}\Big\{(1+\tau)^{-\frac{3}{2}}\Big\}D(t)^{2}\nonumber
\\&\lesssim& (1+t)^{-\frac{3}{2}}N(t)^{2}D(t)^{2}. \label{R-E48}
\end{eqnarray}
Then, it follows from inequalities (\ref{R-E44}), (\ref{R-E45})  and (\ref{R-E48}) that
\begin{eqnarray}
J_{3}\lesssim(1+t)^{-\frac{3}{2}}\|z_{0}\|^{2}_{B^{5/2}_{2,1}}+(1+t)^{-\frac{3}{2}}N(t)^{2}D(t)^{2}+(1+t)^{-3}N(t)^{4}. \label{R-E49}
\end{eqnarray}
Finally, combine (\ref{R-E32}), (\ref{R-E42}) and (\ref{R-E49}) to obtain
\begin{eqnarray}
\|z\|^{2}_{L^{2}}&\lesssim&(1+t)^{-\frac{3}{2}}\|z_{0}\|^{2}_{B^{5/2}_{2,1}\cap\dot{B}^{-\frac{3}{2}}_{2,\infty}}
\nonumber
\\&&+(1+t)^{-\frac{3}{2}}N(t)^{2}D(t)^{2}+(1+t)^{-3}N(t)^{4}, \label{R-E50}
\end{eqnarray}
where we have used the smallness of $\|z_{0}\|_{B^{5/2}_{2,1}}$. This leads to the desired inequality (\ref{R-E27}).
\end{proof}

\noindent \textbf{Proof of Theorem \ref{thm1.3}.}
Note that (\ref{R-E36}), we arrive at
\begin{eqnarray}
D(t)\lesssim\|z_{0}\|_{B^{5/2}_{2,1}}\lesssim I_{1}. \label{R-E52}
\end{eqnarray}
Then the inequality (\ref{R-E27}) can be solved as $N(t)\lesssim I_{1}$ by the standard argument, provided that $I_{1}$ is sufficiently small.
Consequently, the desired decay estimate
\begin{eqnarray}
\|w-w_{\infty}\|_{L^{2}}\lesssim I_{1}(1+t)^{-\frac{3}{4}} \label{R-E53}
\end{eqnarray}
follows immediately. Hence, the proof of Theorem \ref{thm1.3} is complete. \hspace{10mm} $\square$

\section*{Acknowledgments}
J. Xu is partially supported by the National
Natural Science Foundation of China (11471158), the Program for New Century Excellent
Talents in University (NCET-13-0857) and the Fundamental Research Funds for the Central
Universities (NE2015005). The work is also partially supported by
Grant-in-Aid for Scientific Researches (S) 25220702.

\end{document}